\documentclass{gtart_a}
\pdfoutput=1

\usepackage{pinlabel}
\usepackage{enumerate}

\title{Existence of ruled wrappings in hyperbolic $3$--manifolds}

\author{Teruhiko Soma}
\givenname{Teruhiko}
\surname{Soma}
\address{Department of Mathematics and Information Sciences\\
Tokyo Metropolitan University\\\newline
Minami-Ohsawa 1-1, Hachioji\\
Tokyo 192-0397\\Japan}
\email{tsoma@center.tmu.ac.jp}
\urladdr{}

\volumenumber{10}
\issuenumber{}
\publicationyear{2006}
\papernumber{28}
\lognumber{0495}
\startpage{1173}
\endpage{1184}

\doi{}
\MR{}
\Zbl{}

\arxivreference{}  %%% not in arXiv


\received{15 September 2004}
\revised{26 June 2006}
\accepted{08 July 2006}
\published{12 September 2006}
\publishedonline{12 September 2006}
\proposed{Jean-Pierre Otal}
\seconded{David Gabai, Tobias Colding} 
\corresponding{}
\editor{David Gabai}
\version{}

\AtBeginDocument{\let\bar\wbar\let\tilde\wtilde\let\hat\what}
\def\S{Section }
\makeop{Int}

\makeatletter
\def\cnewtheorem#1[#2]#3{\newtheorem{#1}{#3}[section]
\expandafter\let\csname c@#1\endcsname\c@theorem}

\newtheorem{theorem}{Theorem}[section]
\cnewtheorem{cor}[theorem]{Corollary}
\cnewtheorem{lemma}[theorem]{Lemma}
\cnewtheorem{prop}[theorem]{Proposition}

\theoremstyle{definition}
\cnewtheorem{definition}[theorem]{Definition}
\cnewtheorem{remark}[theorem]{Remark}

\makeatother  %  move after \newtheorem block

\numberwithin{equation}{section}

\begin{document}

\begin{asciiabstract}
We present a short elementary proof of an existence theorem of certain
CAT(-1)-surfaces in open hyperbolic 3-manifolds. The main
construction lemma in Calegari and Gabai's proof of Marden's Tameness
Conjecture can be replaced by an applicable version of our theorem.
\end{asciiabstract}

\begin{htmlabstract}
We present a short elementary proof of an existence theorem of certain
CAT(-1)&ndash;surfaces in open hyperbolic 3&ndash;manifolds.  The
main construction lemma in Calegari and Gabai's proof of Marden's
Tameness Conjecture can be replaced by an applicable version of our
theorem.  Finally, we will give a short proof of the conjecture along
their ideas.
\end{htmlabstract}

\begin{abstract}
We present a short elementary proof of an existence theorem of certain
$\mathrm{CAT}(-1)$--surfaces in open hyperbolic $3$--manifolds.  The
main construction lemma in Calegari and Gabai's proof of Marden's
Tameness Conjecture can be replaced by an applicable version of our
theorem.  Finally, we will give a short proof of the conjecture along
their ideas.
\end{abstract}

\maketitle

Agol \cite{ag} and Calegari and Gabai \cite{cg} proved independently that any hyperbolic $3$--manifold $M$ with finitely 
generated fundamental group is homeomorphic to the interior of a compact $3$--manifold.
This is the affirmative answer to Marden's Tameness Conjecture in \cite{ma}.
Subsequently, Choi \cite{cho} gave another proof of the conjecture, similar to Agol's, in the case when $M$ has no 
parabolic cusps.

We are here interested in arguments in \cite{cg}, where the notion of
``shrinkwrapping" was introduced.  Shrinkwrappings play an important
role in their proof.  For the proof of the existence of
shrinkwrappings and that of their $\mathrm{CAT}(-1)$--property,
Calegari and Gabai used very deep and rarefied arguments, which some
readers, including the author, may find difficult to approach.  This
paper is intended to give an elementary proof of part of their proof
by using ruled wrappings instead of shrinkwrappings.

For simplicity, we only consider the case when a hyperbolic $3$--manifold has no parabolic cusps and will prove the 
following theorem.

\begin{theorem}\label{t_1}
Let $N$ be an orientable hyperbolic $3$--manifold without parabolic cusps, $\Delta$ a disjoint union of finitely 
many simple closed geodesics in $N$, and $f\co \Sigma\longrightarrow N$ a $2$--incompressible map rel.\ $\Delta$ from 
a closed orientable surface $\Sigma$ of genus greater than $1$ to $N\setminus \Delta$.
Then there exists a homotopy $F\co \Sigma\times [0,1]\longrightarrow N$ satisfying the following conditions.
\begin{enumerate}[\rm (i)]
\item
$F(x,0)=f(x)$ for any $x\in \Sigma$.
\item
$F(\Sigma\times [0,1))\cap \Delta=\emptyset$.
\item
The map $g\co \Sigma\longrightarrow N$ defined by $g(x)=F(x,1)$ is a $\mathrm{CAT}(-1)$--piecewise ruled map.
\end{enumerate}
\end{theorem}

Here, a continuous map $f\co \Sigma\longrightarrow N$ is said to be \emph{$2$--incompressible\/} in $N$ rel.\ $\Delta$ if 
$f(\Sigma)\cap \Delta=\emptyset$, $f_*\co \pi_1(\Sigma)\longrightarrow \pi_1(N\setminus \Delta)$ is injective, and 
for any simple noncontractible loop $l$ in $\Sigma$ the restriction $f|_l$ is not freely homotopic in $N\setminus 
\Delta$ to a (multiplied) meridian of any component of $\Delta$.
See \fullref{d_1} for the definition of piecewise ruled maps.
We say that a map $g$ satisfying properties (i)--(iii) as above or its image $g(\Sigma)$ is a 
$\mathrm{CAT}(-1)$--\emph{ruled wrapping\/} of $\Delta$ in $N$ homotopic to $f$.
In fact, \fullref{t_1} is a special case of \fullref{p_1}, which corresponds to the main construction 
lemma in \cite{cg}.

In \fullref{S4}, we will give a short proof of Marden's Conjecture along ideas in \cite{cg}.
Our proof is self-contained in the sense that it does not rely on published partial solutions to the conjecture.
It is not hard to see that all arguments and proofs in this paper work also in the case when the ambient manifold 
$N$ has pinched negative curvature and hyperbolic cusps by invoking results in Canary \cite[\S 4]{can}.

\medskip

\textbf{Acknowledgments}\qua
The author would like to thank Brian Bowditch and the referee for helpful comments and suggestions.

\section{Completion of certain hyperbolic metrics}\label{S1}

For a closed subset $A$ in a metric space $(X,d)$, the $r$--neighborhood of $A$ in $X$, $\{y\in X;\, d(y,A)\leq r\}$, 
is denoted by $\mathcal{N}_r(A)$ (or more strictly by $\mathcal{N}_r(A,X)$).
In the case when $A$ is a single point set $\{x\}$, we also set $\mathcal{N}_r(\{x\})=\mathcal{B}_r(x)$.
The link of $x$ in $X$ with radius $r$, $\{y\in X;\, d(y,x)= r\}$, is denoted by $\mathcal{S}_r(x)$.

Let $U$ be a simply connected incomplete hyperbolic $3$--manifold with metric completion $\overline U$ such that each 
component $l$ of $L=\overline U\setminus U$ 
is a geodesic line, and there exists a constant $c>0$ with $\mathrm{dist} (x,y)\geq 3c$ for any points $x,y$ 
contained in distinct components of $L$.
Moreover, we suppose that for any component $l$ of $L$ there exists an infinite cyclic branched covering 
$p_l\co \mathcal{N}_c(l,\overline U)\longrightarrow \mathcal{N}_c(j,\mathbf{H}^3)$
branched over a geodesic line $j$ of $\mathbf{H}^3$, such that the restriction 
$p_l|_{\smash{\mathcal{N}_c(l, \overline U)\setminus l}}$ is a locally isometric covering.
From the definition, $\mathcal{N}_c (l,\overline U)$ is homeomorphic to the quotient space of 
$\mathbf{R}^2\times [0,1]$ by the identification map $a\co \mathbf{R}^2\times \{0\}\longrightarrow \mathbf{R}$ 
defined by $a(x,y,0)=x$.

Suppose that $\sigma$ is a shortest arc in $\overline{U}$ connecting two given points and consisting of two 
hyperbolic segments $\sigma_1,\sigma_2$ with $\sigma\cap L=\sigma_1\cap \sigma_2=\{x\}$.
Let $x_i$ $(i=1,2)$ be the point in $\sigma_i$ with $\mathrm{dist} (x_1,l)=\mathrm{dist}(x_2,l)=s>0$, where $l$ 
is the component of $L$ containing $x$; see \fullref{f_1}\,(a).
%%%%%%%%%%%%%%%%%%%%%%%%%%%%%%%%%%%%%%%%%%%%%%%%%%%%%%%
\begin{figure}[ht!]
\labellist
\small\hair 2pt
\pinlabel $\sigma_1$ [tr] at 150 686
\pinlabel $x_1$ [t] at 172 696
\pinlabel $s$ [t] at 180 733
\pinlabel $x_2$ [bl] at 238 703
\pinlabel $\sigma_2$ [tl] at 255 671
\pinlabel $P_1$ [tl] at 269 678
\pinlabel $P_2$ [bl] at 287 719
\pinlabel $x$ [b] at 215 741
\pinlabel $l$ [b] at 154 755
\pinlabel $\overline{U}$ at 267 751
\pinlabel (a) [t] at 207 660
\pinlabel $\theta_1$ [t] at 398 700
\pinlabel $\theta_2$ [bl] at 404 715 
\pinlabel $\mathbf{H}^2$ at 342 754
\pinlabel (b) [t] at 397 660
\endlabellist
\centering
\includegraphics{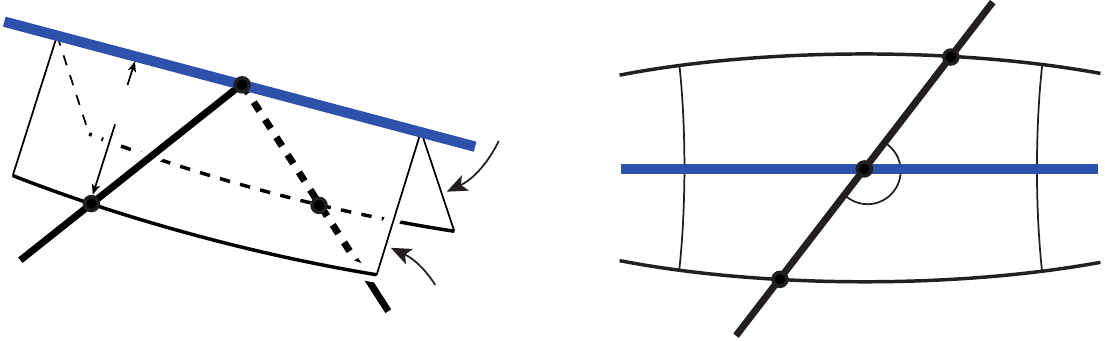}
%\nocolon
\caption{}
\label{f_1}
\end{figure}
%%%%%%%%%%%%%%%%%%%%%%%%%%%%%%%%%%%%%%%%%%%%%%%%%%%%%%%
There exist totally geodesic half planes $P_i$ in $\overline U$ with $\sigma_i\subset P_i$ and $\partial P_i = l$. 
Since the subsegment $\tau$ of $\sigma$ with $\partial \tau=\{x_1,x_2\}$ is the shortest arc in $\overline U$ 
connecting $x_1$ with $x_2$, we have
\begin{equation}\label{eqn_1}
\theta_1+\theta_2=\pi,
\end{equation}
where $\theta_i$ is the angle made by $\sigma_i$ and a fixed ray in $l$ emanating from $x$.
This fact is easily seen by considering the developing of $P_1\cup P_2$ on $\mathbf{H}^2$; see \fullref{f_1}\,(b).

For any $d$ with $0<d\leq c$, the set $\mathcal{B}_d(x,\overline U)$ is homeomorphic to the subset of $\mathbf{R}^3$
$$\{(u,v,w)\in \mathbf{R}^3;\, u^2+v^2+w^2\leq 1, w>0\}\cup \{(u,0,0)\in \mathbf{R}^3;\, -1\leq u\leq 1\}.$$
In particular, $\mathcal{B}_d(x,\overline U)$ is simply connected.
The image $p_l(\mathcal{B}_d(x,\overline U))$ coincides with the hyperbolic ball 
$\mathcal{B}_d(\hat x,\mathbf{H}^3)$, where $\hat x=p_l(x)$.
Rescaling the metric on the boundary $S=\mathcal{S}_d(\hat x,\mathbf{H}^3)$ of the ball, we have the spherical 
metric $\nu$ on $S$ isometric to the unit sphere in the Euclidean $3$--space.
Consider the metric on $\tilde S=\mathcal{S}_d(x,\overline U)$, still denoted by $\nu$, so that the infinite cyclic 
branched covering $p_l|\tilde S\co \tilde S\longrightarrow S$ is locally pathwise isometric.
Here, $p_l|\tilde S$ being \emph{locally pathwise isometric\/} means that 
$\mathrm{length}_\nu (\alpha)=\mathrm{length}_\nu(p_l(\alpha))$ for any rectifiable arc $\alpha$ in $\tilde S$.
One can take $d>0$ so that $\sigma'=\sigma\cap \mathcal{B}_d(x)$ is a geodesic segment in $\mathcal{B}_d(x)$ with 
$\partial \sigma'\subset \tilde S$.
Let $\gamma$ be any rectifiable arc in $\tilde S$ with $\partial \gamma=\partial \sigma'$.
Since $\mathcal{B}_d(x)$ is simply connected, $\gamma$ is homotopic rel.\ $\partial \gamma$ to $\sigma'$ in 
$\mathcal{B}_d(x)$.
Then the following lemma is proved immediately from the equality \eqref{eqn_1} and by checking the situation of 
$\hat \gamma=p_l(\gamma)$ in $S$; see \fullref{f_2}.

\begin{lemma}\label{l_1}
$\mathrm{length}_\nu (\gamma)\geq \pi$.
\end{lemma}

%%%%%%%%%%%%%%%%%%%%%%%%%%%%%%%%%%%%%%%%%%%%%%%%%%%%%%%
\begin{figure}[ht!]
\labellist
\small\hair 2pt
\pinlabel $S$ [br] at 168 750 
\pinlabel $\hat \gamma$ [t] at 201 750
\pinlabel $\hat \sigma$ [tl] at 185 691 
\pinlabel $j$ [bl] at 261 693 
\pinlabel (a) [t] at 202 636
\pinlabel $S$ [br] at 339 750
\pinlabel $\hat \gamma$ [bl] at 395 732 
\pinlabel $\hat \sigma$ [br] at 364 690 
\pinlabel $j$ [bl] at 433 693 
\pinlabel (b) [t] at 375 636
\endlabellist 
\centering
\includegraphics{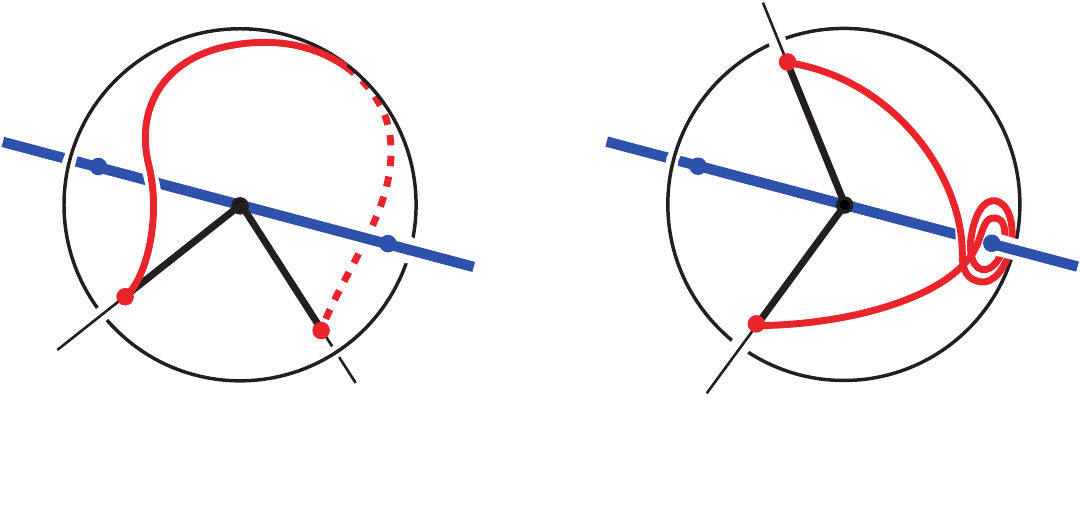}
\caption{$j=p_l(l)$, $\hat \sigma=p_l(\sigma')$.
(b) is the case when $\hat \gamma$ winds around $j$ more than once, but 
$\mathrm{length}_\nu (\gamma)=\mathrm{length}_\nu (\hat \gamma)$ does not exceed $\pi$ very much.}
\label{f_2}
\end{figure}
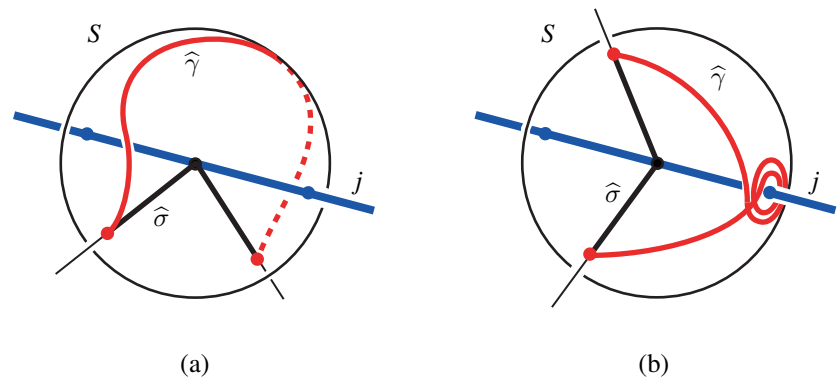
%%%%%%%%%%%%%%%%%%%%%%%%%%%%%%%%%%%%%%%%%%%%%%%%%%%%%%%

For any $k\leq 0$, let $\mathbf{F}(k)$ be a complete Riemannian plane of constant curvature $k$.
A simply connected geodesic metric space $X$ is called a $\mathrm{CAT}(k)$--\emph{space\/} if any geodesic triangle 
$\Delta$ in $X$ is \emph{not thicker\/} than a comparison triangle $\bar \Delta$ in $\mathbf{F}(k)$, that is, for 
any two points $s$ and $t$ in the edges of $\Delta$ and their comparison points $\bar s$ and $\bar t$ in $\bar\Delta$, 
$\mathrm{dist}_X(s,t)\leq \mathrm{dist}_{\mathbf{F}(k)}(\bar s,\bar t)$.
See Bridson--Haefliger \cite{bh} for fundamental properties of such spaces.

\begin{lemma}\label{CAT}
$\overline{U}$ is a $\mathrm{CAT}(-1)$--space.
\end{lemma}
\begin{proof}
By the generalized Cartan--Hadamard Theorem \cite[Chapter II.4, Theorem 4.1(2)]{bh}, it suffices to show that, 
for any point $x\in \overline{U}$, there exists $r>0$ such that $B_r(x)$ is a $\mathrm{CAT}(-1)$--space.
If $x\in U$ then $B_r(x)$ is obviously a $\mathrm{CAT}(-1)$--space if we take $r$ so that $B_r(x)\cap L=\emptyset$.
So suppose that $x\in L$.
Let $\Delta$ be any triangle in $B_c(x)$ with geodesic edges $\gamma_1,\gamma_2,\gamma_3$.
Each $\Int \gamma_i$ either meets $L$ at most one point $p_i$ or is contained in $L$.
From this, we know that the union $\gamma_1\cup \gamma_2\cup \gamma_2$ bounds a triangle $\Delta_0$ in $B_c(x)$ 
consisting of at most three totally geodesic hyperbolic subtriangles.
In fact, when $p_1,p_2,p_3$ exist, the totally geodesic triangle with vertices $p_1,p_2,p_3$ is degenerated to 
a geodesic segment in $L$.
By \fullref{l_1}, the internal angle of $\Delta_0$ at any $p_i$ is at least $\pi$.
This shows that $\Delta_0$ is not thicker than a comparison triangle in $\mathbf{H}^2$.
It follows that $B_c(x)$ is a $\mathrm{CAT}(-1)$--space.
\end{proof}

It is not hard to see that any geodesic segment in $\overline U$ is a broken line consisting of finitely 
many hyperbolic segments, and any vertex of the broken line other than its end points lies in $L$.
By \fullref{CAT}, a geodesic segment $\sigma$ in $\overline U$ connecting two given points is uniquely 
determined.
Moreover, $\sigma$ varies continuously with its end points.

Let $Z$ be an incomplete hyperbolic $3$--manifold such that the total space of the universal covering 
$q\co U\longrightarrow Z$ has the induced metric as above.
We suppose moreover that for the metric completion $\overline Z$ each component $l$ of $\overline Z\setminus Z$ 
is either a geodesic line or a geodesic loop.
That is, no component of $\overline Z\setminus Z$ is a one point set.
Then $q\co U\longrightarrow Z$ is extended to a locally pathwise isometric map $\bar q\co \overline U\longrightarrow 
\overline Z$.
Note that the frontier $\{x\in \overline Z;\, \mathrm{dist}(x,l)=c\}$ of $\mathcal{N}_c(l,\overline Z)$ in 
$\overline Z$ is homeomorphic to either $\mathbf{R}^2$ or an open annulus or a torus.

\begin{remark}\label{r_1}
Even in the case when $\mathcal{N}_c(l,\overline Z)$ is homeomorphic to a solid torus, we always suppose that 
\emph{homotopies in\/} $\overline Z$ starting from a continuous map $f\co A\longrightarrow Z$ never cross $l$ (possibly 
they touch $l$), where $A$ is a manifold of dimension less than three.
In other words, we only consider homotopies $F\co A\times [0,1]\longrightarrow \overline Z$ which can be covered by a 
map $\tilde F\co \tilde A\times [0,1]\longrightarrow \overline U$, where $\tilde A$ is the universal covering space 
of $A$.
\end{remark}

\begin{definition}\label{d_1}
Let $f\co \Sigma\longrightarrow \overline Z$ be a continuous map from a closed orientable surface $\Sigma$.
Suppose that $\Sigma$ admits a cell decomposition $K$ consisting of finitely many triangular $2$--cells.
We say that $f$ is a \emph{piecewise ruled map\/} with respect to $K$ if (i) for each edge of $e$ of $K$, $f(e)$ 
is a broken line consisting of finitely many hyperbolic segments, and (ii) for each $2$--cell $F$ of $K$, $f(F)$ is 
a ruled triangle based on a single vertex.
Since the Gaussian curvature at any smooth point of a ruled surface is nonpositive, the (intrinsic) curvature of 
$\Sigma$ at any smooth point is at most $-1$. 
This map $f$ is called a $\mathrm{CAT}(-1)$--\emph{piecewise ruled map\/} if the cone-angle of $\Sigma$ at any 
singular point is at least $2\pi$.
\end{definition}

\section[Applicable version of Theorem 0.1]{Applicable version of \fullref{t_1}}\label{S2}

Throughout this section, we assume that all hyperbolic $3$--manifolds and surfaces are orientable.

Let $N$ be a complete hyperbolic $3$--manifold $N$ without parabolic cusps, and $W$ a $3$--dimensional compact connected 
submanifold of $N$.
Consider a link $\Delta$ in $\Int W$ consisting of finitely many disjoint simple closed geodesics in $N$.
Let $p\co X\longrightarrow W$ be the covering of $W$ associated to a finitely generated subgroup of $\pi_1(W)$.
Here we make the following assumptions, which correspond to those in the main construction lemma \cite[Lemma 2.3]{cg}.
\begin{enumerate}[\rm (i)]
\item
$\partial W$ is incompressible in $N\setminus \Delta$.
\item
There exists a union $\hat \Delta$ of components of $p^{-1}(\Delta)$ such that the restriction $p|\hat \Delta\co 
\hat \Delta\longrightarrow \Delta$ is a homeomorphism.
\item
There exists a continuous map $f\co \Sigma\longrightarrow X$ from a closed surface $\Sigma$ of genus $m >1$ which 
is $2$--incompressible in $X$ rel.\ $\hat \Delta$.
\end{enumerate}

Set $X^\circ =X\setminus p^{-1}(\Delta)$.
The fundamental group $\pi_1(X^\circ)$ may be infinitely generated.
By (i), the restriction $p^\circ =p|X^\circ\co X^\circ\longrightarrow W\setminus \Delta\subset N\setminus \Delta$ 
is $\pi_1$--injective.
Thus, we may assume that $X^\circ$ is a subset of the total space of the covering $q\co Y^\circ \longrightarrow N
\setminus \Delta$ associated to the subgroup $p^\circ_*(\pi_1(X^\circ))$ of $\pi_1(N\setminus \Delta)$ and the 
inclusion $i\co X^\circ\longrightarrow Y^\circ$ is a homotopy equivalence.
Since $\partial X^\circ=\partial X=p^{-1}(\partial W)$ is a deformation retract of $Y^\circ \setminus 
\Int X^\circ$, condition (iii) implies that $f\co \Sigma\longrightarrow Y$ is $2$--incompressible in 
$Y$ rel.\ $\hat \Delta$, where $Y=X\cup (Y^\circ \setminus \Int X^\circ)=X\cup Y^\circ$.
The complement $Z=Y\setminus \hat \Delta$ has the induced incomplete metric as was studied in \fullref{S1}.
Let $\overline Z$ be the metric completion of $Z$.

With the notation and assumptions as above, we will prove the following proposition.

\begin{prop}\label{p_1}
There exists a homotopy $F\co \Sigma\times [0,1]\longrightarrow \overline Z$ which never crosses $\hat \Delta$ and 
connects $f$ with a $\mathrm{CAT}(-1)$--ruled wrapping $g\co \Sigma\longrightarrow \overline Z$ of $\hat \Delta$.
\end{prop}

\begin{proof}%[Proof of \fullref{p_1}]
Let $c_1,\ldots, c_{3m-3}$ be mutually disjoint simple loops in $\Sigma$ which define a pants decomposition of 
$\Sigma$.
Consider a cell decomposition $K$ of $\Sigma$ consisting of triangular $2$--cells and such that each vertex of $K$ 
is contained in $c_1\cup\cdots\cup c_{3m-3}$.
If necessary by deforming $f$ by a homotopy in the sense of \fullref{r_1} we may assume that each $f(c_i)$ 
is a closed geodesic in $\overline Z$, and $f(e)$ is a geodesic segment in $\overline Z$ for any edge $e$ of $K$ 
not contained in $c_1\cup\cdots\cup c_{3m-3}$.
In fact, $f(c_i)$ is the image of an axis of a hyperbolic transformation on the metric completion $\overline U$ 
of the universal covering space of $Z$.
See for example \cite[Theorem 6.8\,(1)]{bh}.
For any $2$--cell $F$ of $K$, take a vertex $v_0$ and the opposite edge $e_0$.
Then, $f|F$ can be homotoped rel.\ $\partial F$ to a map $g|F$ such that $g(F)$ is a ruled triangle consisting of 
all geodesic segments connecting $f(v_0)$ with points of $f(e_0)$.
These $g|F$ define a piecewise ruled map $g\co \Sigma\longrightarrow \overline Z$ homotopic to $f$.
From our construction of $g$, for any singular point $g(v)$ of $g(\Sigma)$, there exists an arc $\alpha$ in 
$\Sigma$ with $\Int \alpha\ni v$ and such that $g(\alpha)$ is a geodesic segment in $\overline Z$.
If $g(v)$ is not an element of $\overline Z\setminus Z$, then it is easily seen that the cone-angle of $\Sigma$ 
at $v$ is at least $2\pi$.
So we may assume that $g(v)$ is contained in a component $l$ of $\overline Z\setminus Z$.
For any sufficiently small $d>0$, $\alpha$ divides the circle $\mathcal{S}_d(v,\Sigma)$ into two arcs $\gamma_1$, 
$\gamma_2$.
By \fullref{l_1}, the $\nu$--length of $g(\gamma_i)$ $(i=1,2)$ in $\mathcal{S}_d(g(v),\overline Z)$ is at 
least $\pi$.
Thus, the cone-angle of $\Sigma$ at $v$ is at least $2\pi$.
This shows that $g$ is a $\mathrm{CAT}(-1)$--ruled wrapping of $\hat \Delta$ in $\overline Z$.
\end{proof}

Note that \fullref{t_1} is proved quite similarly to \fullref{p_1} by considering $(N,\Delta)$ 
instead of $(\overline Z,\hat \Delta)$.

\section{Compact cores and end reductions}\label{S3}

A $3$--manifold $X$ is \emph{topologically tame\/} if there exists an embedding $f\co X\longrightarrow Y$ into a compact 
manifold $Y$ with $f(X)\supset \Int Y$.
Throughout this section, we suppose that $M$ is an orientable, open, irreducible and connected $3$--manifold with 
finitely generated fundamental group.
An end $\mathcal{E}$ of $M$ is said to be \emph{topologically tame\/} if there exits a closed neighborhood of 
$\mathcal{E}$ in $M$ homeomorphic to $S\times [0,\infty)$ for some closed connected surface $S$.
It is easily seen that the open $3$--manifold $M$ is topologically tame if and only if each end of $M$ is so.

Scott \cite{sc} proved that $M$ contains a $3$--dimensional submanifold $C$, called a \emph{compact core\/} of $M$, 
such that the inclusion $i\co C\longrightarrow M$ is a homotopy equivalence.
Let $S$ be the component of $\partial C$ facing an end $\mathcal{E}$ of $M$, and $p\co \tilde M\longrightarrow M$ 
the covering associated with the image of $\pi_1(S)$ in $\pi_1(M)$.
There exists a compact core $\tilde C$ of $\tilde M$ such that $\partial \tilde C$ has a component $\tilde S$ 
mapped onto $S$ homeomorphically by $p$.
The manifold $\tilde C$ is a \emph{compression body\/}, ie, it is homeomorphic to $E\cup h_1\cup\cdots \cup h_m$ 
where $E$ is either a $3$--ball or $F\times [0,1]$ for some closed surface $F$ consisting of nonspherical components 
and the $h_i$'s are $1$--handles attached to one side of $E$.
In particular, when $E$ is a $3$--ball, the compression body $\tilde C$ is a handlebody.
Note that the end $\tilde{\mathcal{E}}$ of $\tilde M$ faced by $\tilde S$ is topologically tame if and only if 
$\mathcal{E}$ is also.

Let $\Delta=\delta_1\cup \cdots\cup \delta_{i_0}$ be an $i_0$--component link in the compression body $\tilde C$ 
such that $[\delta_k]$ $(k=1,\cdots,i_0-1)$ form a basis for the free abelian group $H_1(\tilde C,\mathbf{Z})$ 
and $[\delta_{i_0}]=[\delta_1]+\cdots+[\delta_{i_0-1}]$.
An advantage of considering compression bodies is that any nontrivial free decomposition of $\pi_1(\tilde C)$ 
induces a nontrivial decomposition of $H_1(\tilde C;\mathbf{Z})$.
In particular, this implies that the link $\Delta$ is \emph{algebraically disk-busting\/}, that is, for any 
nontrivial free decomposition $A*B$ of $\pi_1(\tilde C)$, there exists a component $\delta_k$ of $\Delta$ such 
that the element of $\pi_1(\tilde C)$ represented by $\delta_k$ is neither conjugate into $A$ nor $B$.

Some results in Myers \cite{my} concerning end reductions play an important role in the proof of \fullref{t_2}.
The paper is useful also as an expository article on end reductions.
A compact, connected, $3$--dimensional submanifold $R$ of $M$ is \emph{regular\/} if $M\setminus R$ is irreducible and 
the closure of any component of $M\setminus R$ in $M$ is not compact.
Let $\Delta$ be a link in $M$ each component of which is noncontractible in $M$.
An open submanifold $V$ of $M$ containing $\Delta$ is called an \emph{end reduction\/} of $M$ at $\Delta$ if it 
satisfies the following conditions.
\begin{enumerate}[\rm (i)]
\item
No component of $M\setminus V$ is compact.
\item
There exists a sequence $\{R_n\}$ of regular submanifolds of $M$ with $\Delta\subset R_1$, $R_n\subset 
\Int R_{n+1}$, $V=\bigcup_n R_n$ and such that $\partial R_n$ is incompressible in $M\setminus \Delta$.
\item
$V$ satisfies the \emph{engulfing property\/} at $\Delta$, that is, for any regular submanifold $N$ of $M$ with 
$\Delta \subset \Int N$ such that $\partial N$ is incompressible in $M\setminus \Delta$, $V$ is ambient 
isotopic rel.\ $\Delta$ to a manifold containing $N$.
\end{enumerate}
We refer to Brin--Thickstun \cite{bt} for the existence and uniqueness up to isotopy of end reductions.
According to Myers \cite[Theorem 9.2]{my}, if the link $\Delta$ is algebraically disk-busting, then an end 
reduction $V$ of $M$ at $\Delta$ is connected and the homomorphism $i_*\co \pi_1(V)\longrightarrow \pi_1(M)$ 
induced from the inclusion is an isomorphism.

\section{Proof of Marden's Conjecture}\label{S4}

Our proof of Marden's Conjecture is based on that of Calegari--Gabai \cite{cg}, but the importance of the 
disk-busting property is suggested by Agol \cite{ag}.
For simplicity, we only consider hyperbolic $3$--manifolds without parabolic cusps.
It is not hard to modify our argument for the case of manifolds with parabolic cusps.

\begin{theorem}[Marden's Tameness Conjecture]\label{t_2}
Let $N$ be an orientable hyperbolic $3$--manifold without parabolic cusps.
If $\pi_1(N)$ is finitely generated, then $N$ is topologically tame.
\end{theorem}

First of all, we fix the setting for the proof.
It suffices to show that each end $\mathcal{E}$ of $N$ is topologically tame.
As was seen in \fullref{S3}, we may assume that a compact core $C$ of $N$ is a compression body.
Let $S$ be the component of $\partial C$ facing $\mathcal{E}$.
If $\mathcal{E}$ is \emph{geometrically finite\/}, that is, $C$ is locally convex in $S$, then it is well known 
that $\mathcal{E}$ is topologically tame, for example see Marden \cite{ma}.
So we may assume that $\mathcal{E}$ is not geometrically finite.
Then Bonahon \cite{bo} shows that there exists a sequence $\{\delta_i\}$ of closed geodesics in $N$ exiting 
$\mathcal{E}$.
If necessary, by adding finitely many closed geodesics to $\{\delta_i\}$ one can suppose that $\Delta_i=\delta_1
\cup \cdots\cup \delta_i$ is algebraically disk-busting for all $i$ not less than some fixed integer $i_0>0$.
If necessarily, by slightly deforming the hyperbolic metric in a small neighborhood of $\bigcup_i \delta_i$ in 
$N$ we may assume that the closed geodesics $\delta_i$ are simple and mutually disjoint, ie, each $\Delta_i$ 
is a link in $N$.
In fact, the resulting metric is no longer hyperbolic but pinched negatively curved.
However, all the results concerting hyperbolic manifolds which we need, eg \fullref{p_1} in 
\fullref{S2}, still hold under this metric if we replace $\mathrm{CAT}(-1)$ with $\mathrm{CAT}(-a^2)$, 
where $-a^2$ is the supremum of sectional curvatures of $N$ with respect to the new metric.
We refer to Canary \cite[Sections 4 and 5]{can} for standard arguments on such a metric deformation.

For any $i\geq i_0$, let $V_i$ be an end reduction of $N$ at $\Delta_i$.
By \cite{my}, the $\pi_1$--homomorphism induced from the inclusion $V_i\longrightarrow N$ is an isomorphism.
It follows that a compact core $C_i$ of $V_i$ is also a compact core of $N$, and each $\delta_k$ $(k=1,\ldots,i)$ 
is freely homotopic in $V_i$ to a loop of $C_i$.
By property (ii) of the end reduction $V_i$, there exists a regular submanifold $W_i$ of $V_i$ containing 
both $C_i$ and the traces of these free homotopies in $C_i$ and such that $\partial W_i$ is incompressible in 
$N\setminus \Delta_i$.
In fact, $R_n$ in \fullref{S3} with sufficiently large $n$ satisfies the properties of $W_i$.
Since the inclusion $C_i\longrightarrow W_i\longrightarrow N$ is $\pi_1$--isomorphic, $\pi_1(C_i)$ can be regarded 
as a subgroup of $\pi_1(W_i)$.
Consider the covering $p_i\co X_i\longrightarrow W_i$ associated to $\pi_1(C_i)\subset \pi_1(W_i)$.
Let $\hat\delta_k$ be a component of $p^{-1}_i(\delta_k)$ such that $p_i|\hat \delta_k\co \hat\delta_k\longrightarrow 
\delta_k$ is homeomorphic, and let $\hat \Delta_i=\hat\delta_1\cup\cdots\cup\hat\delta_i$.
The component $S_i$ of $\partial C_i$ facing $\mathcal{E}$ has a lift $S_i^\bullet$ to $X_i$ which is also a 
boundary component of a compact core of $X_i$.

\begin{lemma}\label{Tame}
$X_i$ is topologically tame.
\end{lemma}
\begin{proof}
The claim in \cite[page 426]{cg} shows that $W_i$ is an atoroidal manifold such that $\partial W_i$ has a component 
of genus greater than $1$.
This fact together with Canary \cite{can} proves that $X_i$ is topologically tame.

Here, we will give an another proof without invoking the atoroidality of $W_i$.
First we outline the proof.
Divide the covering $p_i\co X_i\longrightarrow W_i$ to restricted coverings associated with a torus decomposition 
of $W_i$.
By finite generation, all but finitely many restrictions are universal coverings.
By a result in Waldhausen \cite{wa}, the total spaces of such universal coverings are topologically tame.
Other total spaces cover atoroidal Haken manifolds with non-torus boundary component, and hence they are 
topologically tame.
Since $X_i$ is obtained by attaching topologically tame manifolds to each other along simply connected boundary 
components, Simon's Combination Theorem \cite{sim} shows that $X_i$ is also topologically tame.

More precisely, let $\mathcal{T}_i=T_1\cup\cdots\cup T_m$ be a maximal union of mutually disjoint and nonparallel 
incompressible tori in $\Int W_i$.
Since $W_i$ is regular and $N$ is atoroidal and irreducible, each $T_j$ bounds a compact manifold $A_j$ which is 
contained in a $3$--ball in $N$ and homeomorphic to the exterior of a nontrivial knot in $S^3$.
Either any two $A_j$ are mutually disjoint or one of them contains the other.
Set $\mathcal{A}=A_1\cup\cdots\cup A_m$.
Note that, for any component $\hat A$ of $p_i^{-1}(\mathcal{A})$, the image $\mathrm{inc}_i\circ p_i(\hat A)$ 
is contained in a $3$--ball in $N$, where $\mathrm{inc}_i\co W_i\longrightarrow N$ is the inclusion.
Since $\mathrm{inc}_i\circ p_i\co X_i\longrightarrow N$ is $\pi_1$--isomorphic, it follows that $\hat A$ is simply 
connected.
Then $\hat A$ is topologically tame by \cite[Theorem 8.1]{wa}.
Since each component of $\partial \hat A$ is simply connected, $p_i^{-1}(\partial \mathcal{A})$ induces a 
free decomposition of $\pi_1(X_i)$.
The classical Grushko Theorem implies that the fundamental group of any component $\hat B$ of $
p_i^{-1}(W_i\setminus \Int \mathcal{A})$ is finitely generated.
Since $W_i\setminus \Int \mathcal{A}$ is an atoroidal Haken manifold such that one of the boundary 
components has genus greater than $1$, $\Int \hat B$ is topologically tame by \cite[Proposition 3.2]{can}.
In the present case, it is not hard to show that $\hat B$ is also topologically tame.
See Soma \cite{so} for a more general case.
Finally, Theorem 3.1 in \cite{sim} implies that $X_i$ is topologically tame.
\end{proof}

\begin{proof}[Proof of \fullref{t_2}]
By \fullref{Tame}, $X_i$ is realized as the interior of a compact manifold $\overline X_i$.
Let $\overline S_i$ be the component of $\partial \overline X_i$ facing $S_i^\bullet$ in $X_i$, 
and $\hat S_i$ a closed surface in $\Int X_i$ obtained by a small isotopy of $\overline S_i$ in 
$\overline X_i$.
We show that $\hat S_i$ is $2$--incompressible in $X_i$ rel.\ $\hat \Delta_i$.
If not, there would exist a compressing disk $D$ for $\overline S_i$ in $\overline X_i$ such that the 
intersection $D\cap \hat \Delta_i$ consists of at most one point.
If $D$ separates $X_i$, then $\pi_1(X_i)$ has a nontrivial free decomposition $A*B$ each factor of which is 
isomorphic to the fundamental group of a component of $X_i\setminus D$.
Otherwise, $\pi_1(X_i)$ is isomorphic to $A*\mathbf{Z}$ with $A=\pi_1(X_i\setminus D)$.
When $D\cap \hat \Delta_i=\emptyset$, any element of $\pi_1(X_i)$ represented by a component of $\hat \Delta_i$ is 
conjugate into one of the factors.
When $D\cap \hat \delta_j\neq \emptyset$ for some component $\hat \delta_j$ of $\hat \Delta_i$, one can suppose 
that the cyclic factor $\mathbf{Z}$ is generated by the element represented by $\hat \delta_j$.
Any element represented by a component of $\hat \Delta_i\setminus \hat\delta_j$ is conjugate into $A$.
It follows that $\hat \Delta_i$ is not algebraically disk-busting in $X_i$.
Since $(\mathrm{inc}_i\circ p_i)_*\co \pi_1(X_i)\longrightarrow \pi_1(N)$ is isomorphic and 
$\mathrm{inc}_i\circ p_i(\hat \Delta_i)= \Delta_i$, the link $\Delta_i$ would not be algebraically 
disk-busting in $N$, a contradiction.
One can show similarly that $\hat S_i$ is $2$--incompressible in $X_i$ also rel.\ $\hat \Delta_{i;\,i_0}=
\hat \Delta_i\cap p_i^{-1}(\Delta_{i_0})$.

For any $i\geq i_0$, let $\bar q_i\co \overline Z_i\longrightarrow N$ be the locally pathwise isometric map 
extending the covering $q_i\co Y_i^\circ \longrightarrow N\setminus \Delta_i$ given in \fullref{S2} which satisfies 
$q_i=p_i$ on $X_i\setminus p_{\smash{i}}^{-1}(\Delta_i)$.
Note that $\overline Z_i$ is the metric completion of $Z_i=X_i\cup Y_i^\circ \setminus \hat \Delta_i$.
By \fullref{p_1}, $\hat S_i$ is homotopic in $\overline Z_i$ to a $\mathrm{CAT}(-a^2)$--ruled 
wrapping $\hat\Sigma_i$ without crossing $\hat\Delta_i$.
The image $\Sigma_i=\bar q_i(\hat\Sigma_i)$ is also a $\mathrm{CAT}(-a^2)$--surface homotopic in $N$ to $S_i$.

Since the components $S_i$ of $\partial C_i$ are homeomorphic to each other, all $\hat \Sigma_i$ are closed 
surfaces of the same genus.
Fix an $\varepsilon >0$ less than the Margulis constant for $N$ and so that 
$\mathcal{N}_{2\varepsilon}(\Delta_{\smash{i_0}},N)$ is a tubular neighborhood of $\Delta_{\smash{i_0}}$ in $N$.
Let $l$ be any simple noncontractible loop in $\hat \Sigma_i$ of length less than $\varepsilon$.
If $l$ were contractible in $\overline Z_i$, then $l$ would either bound a disk in $\overline Z_i$ disjoint 
from $\hat \Delta_{i;\,i_0}$ or be contained in $\mathcal{N}_{2\varepsilon}(\hat \Delta_{i;\,i_0},\overline Z_i)
\setminus \hat\Delta_{i;\,i_0}$.
In either case, this contradicts that $\hat \Sigma_i$ is $2$--incompressible rel.\ $\hat \Delta_{i;\,i_0}$.
Thus $l$ is not contractible in $\overline Z_i$, and hence $\bar q_i(l)$ is contained in the 
$\varepsilon$--thin part $N_{\mathrm{thin}(\varepsilon)}$ of $N$.
From this fact together with Bounded Diameter Lemma \cite[Lemma 1.15]{cg} for $\mathrm{CAT}(-a^2)$--surfaces, 
we know that the diameter of any component of $\Sigma_i\setminus N_{\mathrm{thin}(\varepsilon)}$ is less than 
a constant independent of $i$.
Let $\hat \alpha_i$ be a ray in $\overline Z_i$ beginning at $\smash{\hat \delta}_i$ and covering a proper ray 
$\alpha_i$ in $N$ such that the sequence $\{\alpha_i\}$ exits $\mathcal{E}$.
Since the algebraic intersection number of $\hat\alpha_i$ with $\hat S_i$ is one, $\hat\Sigma_i\cap \hat\alpha_i$ 
and hence $\Sigma_i\cap \alpha_i$ are not empty.
This shows that $\{\Sigma_i\}$ exits $\mathcal{E}$.
Under the present situation, the tameness of $\mathcal{E}$ is proved by standard arguments in hyperbolic geometry, 
for example see \cite{bo,can,sou} or Tameness Criteria in \cite[\S 6]{cg}.
However, the case when $C_i$ is a handlebody is exceptional.
As is pointed out in the paragraph following the statement of Theorem 3 in Souto \cite{sou}, we need furthermore 
to show that $[\Sigma_i]\neq 0$ in $H_2(N\setminus \delta_1;\mathbf{Z})$ for guaranteeing that any 
surface homologous to $\Sigma_i$ in $N\setminus C$ excises the end $\mathcal{E}$ from $N$, where $C$ is a compact 
core of $N$ with $\Int C\supset \delta_1$ and $C\cap \Sigma_i=\emptyset$.

Suppose that a compact core of $N$ and hence any $\overline X_i$ are handlebodies.
Let $J$ be a tubular neighborhood of $\delta_1$ in $N$ and $\hat J$ 
a lift of $J$ to $X_i$ containing $\hat \delta_1$.
For $i$ sufficiently large, $J\cap \Sigma_i$ is empty and hence $\bar q_i^{-1}(J)\cap 
\hat \Sigma_i \subset \bar q_i^{-1}(J\cap \Sigma_i)=\emptyset$.
The closure $K$ of the union of bounded components of $\bar Z_i\setminus (\hat \Sigma_i\cup \hat J)$ is compact.
As $\bar q_i^{-1}(\delta_1)$ is disjoint from $\hat\Sigma_i\cup \partial \hat J\supset \partial K$, the preimage
$\bar q_i^{-1}(\delta_1)$ contains no line components meeting $K$ nontrivially.
Since $\bar q_i^{-1}(\delta_1)\setminus p_i^{-1}(\delta_1)\subset \partial \overline Z_i$, if  
$\bar q_i^{-1}(\delta_1)\cap K$ were nonempty, then each component $\smash{\hat\delta_1'}$ of the intersection would be  
a loop component of $p_i^{-1}(\delta_1)$. 
Since $(\mathrm{inc}_i\circ p_i)_*\co \pi_1(X_i) \longrightarrow \pi_1(N)$ is an isomorphism, 
$\hat \delta_1'$ is freely homotopic in $X_i\subset \overline Z_i$ to $\smash{\hat \delta_1}$ up to multiplicity.
Then there would exist a $\mathrm{CAT}(-a^2)$--piecewise ruled annulus in $\overline Z_i$ with the geodesic boundary 
$\hat\delta_1\cup \hat \delta_1'$, which is a contradiction by the Gauss--Bonnet Theorem.
It follows that $\bar q_i^{-1}(\delta_1)\cap K=\emptyset$,
and so $[\partial \smash{\hat J}]+[\smash{\hat \Sigma}_i]=0$ in $H_2(\overline Z_i\setminus 
\bar q_i^{-1}(\delta_1);\mathbf{Z})$.
This shows that $[\Sigma_i]=-[\partial J]\neq 0$ in $H_2(N\setminus \delta_1;\mathbf{Z})$.
\end{proof}

\bibliographystyle{gtart}
\bibliography{link}

\begin{thebibliography}{}
\providecommand\bibmarginpar{\leavevmode\marginpar}
\def\urlstyle#1{{\tt #1}}

\bibitem{ag}
\textbf{I Agol}, \emph{Tameness of hyperbolic 3--manifolds}
  \xox{arXiv}{math.DG/0405568}

\bibitem{bo}
\textbf{F Bonahon},
  \href{http://links.jstor.org/sici?sici=0003-486X(198607)2:124:1%3C71:BDVHDD%%
3E2.0.CO%3B2-9} {\emph{Bouts des vari\'et\'es hyperboliques de dimension 3}},
  Ann. of Math. $(2)$ 124 (1986) 71--158 \xox{MR}{847953}

\bibitem{bh}
\textbf{M\,R Bridson}, \textbf{A Haefliger}, \emph{Metric spaces of
  non-positive curvature}, Grundlehren series 319, Springer, Berlin (1999)
  \xox{MR}{1744486}

\bibitem{bt}
\textbf{M\,G Brin}, \textbf{T\,L Thickstun},
  \href{http://dx.doi.org/10.1016/0040-9383(87)90062-0} {\emph{Open,
  irreducible 3--manifolds which are end 1--movable}}, Topology 26 (1987)
  211--233 \xox{MR}{895574}

\bibitem{cg}
\textbf{D Calegari}, \textbf{D Gabai},
  \href{http://dx.doi.org/10.1090/S0894-0347-05-00513-8} {\emph{Shrinkwrapping
  and the taming of hyperbolic 3--manifolds}}, J. Amer. Math. Soc. 19 (2006)
  385--446 \xox{MR}{2188131}

\bibitem{can}
\textbf{R\,D Canary},
  \href{http://links.jstor.org/sici?sici=0894-0347(199301)6:1%3C1:EOH3%3E2.0.C%
O%3B2--H} {\emph{Ends of hyperbolic 3--manifolds}}, J. Amer. Math. Soc. 6
  (1993) 1--35 \xox{MR}{1166330}

\bibitem{cho}
\textbf{S Choi}, \emph{Drilling cores of hyperbolic 3--manifolds to prove
  tameness}, preprint (2004) \xox{arXiv}{math.GT/0410381}

\bibitem{ma}
\textbf{A Marden},
  \href{http://links.jstor.org/sici?sici=0003-486X(197405)2:99:3%3C383:TGOFGK%%
3E2.0.CO%3B2--Z} {\emph{The geometry of finitely generated kleinian groups}},
  Ann. of Math. $(2)$ 99 (1974) 383--462 \xox{MR}{0349992}

\bibitem{my}
\textbf{R Myers}, \href{http://dx.doi.org/10.2140/gt.2005.9.971} {\emph{End
  reductions, fundamental groups, and covering spaces of irreducible open
  3--manifolds}}, Geom. Topol. 9 (2005) 971--990 \xox{MR}{2140996}

\bibitem{sc}
\textbf{G\,P Scott}, \emph{Compact submanifolds of 3--manifolds}, J. London
  Math. Soc. $(2)$ 7 (1973) 246--250 \xox{MR}{0326737}

\bibitem{sim}
\textbf{J Simon},
  \href{http://projecteuclid.org/getRecord?id=euclid.mmj/1029001718}
  {\emph{Compactification of covering spaces of compact 3--manifolds}},
  Michigan Math. J. 23 (1976) 245--256 (1977) \xox{MR}{0431176}

\bibitem{so}
\textbf{T Soma}, \emph{Covering 3--manifolds with almost compact interior},
  Quart. J. Math. Oxford Ser. $(2)$ 44 (1993) 345--353 \xox{MR}{1240478}

\bibitem{sou}
\textbf{J Souto}, \href{http://dx.doi.org/10.1016/j.top.2004.09.001} {\emph{A
  note on the tameness of hyperbolic 3--manifolds}}, Topology 44 (2005)
  459--474 \xox{MR}{2114957}

\bibitem{wa}
\textbf{F Waldhausen},
  \href{http://links.jstor.org/sici?sici=0003-486X(196801)2:87:1%3C56:OI3WAS%3%
E2.0.CO%3B2--A} {\emph{On irreducible 3--manifolds which are sufficiently
  large}}, Ann. of Math. $(2)$ 87 (1968) 56--88 \xox{MR}{0224099}

\end{thebibliography}

\end{document}